\documentclass[10pt]{article}
\usepackage{amsthm,amsmath}
\usepackage{amssymb}

 \newtheorem{thm}{Theorem}[section]

 \newtheorem{defn}[thm]{Definition}
 \newtheorem{rem}[thm]{Remark}

\title{Opial's inequality in $q$-Calculus revisited}

\author{Tatjana Z. Mirkovi\'c$^1$, Slobodan B. Tri\v ckovi\'c$^2$ Miomir S. Stankovi\'c$^3$,\\[0.5mm]
$^1$\small School of Electrical Engineering, Belgrade, Serbia,\\[0.5mm]
$^2$\small University of Ni\v s, Department of Mathematics, Ni\v s, Serbia\\[0.5mm]
$^3$\small Mathematical Institute of the Serbian Academy of Sciences and Arts, Belgrade, Serbia}

\date{}%

\begin{document}

\maketitle

\begin{abstract} We have fundamentally corrected the proofs of the theorems from our paper \cite{q-O}
by giving an entirely different approach, using quite a simple method based on applications of some
elementary inequalities, well-known H\"older's inequality and the Gauchman $q$-restricted integral.
\end{abstract}




\section{Introduction and preliminaries}

In the recent paper \cite{q-O} a generalization of the Opial integral inequality
\begin{eqnarray}\label{56586}
\int_0^h |f(x)f'(x)|dx\leq\frac{h}{4}\int_0^h {(f'(x))^2 dx}
\end{eqnarray}
in $q$-calculus was given. Here we eliminate some inaccuracies by simplifying and modifying the proofs of the theorems.

First of all, we present necessary definitions and facts from the $q$-calculus, where $q$ is a
real number satisfying $0 < q < 1$, and $q$-natural number is defined by
$$
[n]_q=\frac{1-q^n}{1-q}=q^{n-1}+\cdots+q+1,\dots,n\in\mathbb N.
$$

\begin{defn}\label{df1} Let $f$ be a function defined on an interval $(a,b)\subset\mathbb R$,
so that $qx\in(a,b)$ for all $x\in(a,b)$. For $0 < q <1$, we define the
$q$-derivative as
\begin{equation}\label{q-der}
(D_q f)(x)=\frac{f(x)-f(qx)}{x-qx},\quad x\ne0;\quad D_q f(0)=\lim_{x\to 0}D_q f(x).
\end{equation}
\end{defn}

In the  paper \cite{Jackson}, Jackson defined $q$-integral, which in the $q$-calculus
bears his name.

\begin{defn} The $q$-integral on $[0,a]$ is
$$
\int_0^a f(x)d_q x=a(1-q)\sum_{j=0}^\infty q^j f(aq^j).
$$
\end{defn}

On this basis, in the same paper, Jackson defined an integral on $[a,b]$
\begin{equation}\label{jackson-ab}
\int_a^b f(x)d_q x=\int_0^b f(x)d_q x-\int_0^a f(x)d_q x,
\end{equation}

For a positive integer $n$ and $a=bq^n$, using the left-hand side integral
of (\ref{jackson-ab}), in the paper \cite{Gauchman}, Gauchman introduced
the $q$-restricted integral
\begin{equation}\label{56}
\int_a^b f(x)d_q x=\int_{bq^n}^b f(x)d_q x=b(1 -q)\sum_{j=0}^{n-1}q^j f(q^j b).
\end{equation}

\begin{defn}The real function $f$ defined on $[a,b]$ is called $q$-increasing
{\rm (}$q$-decreasing{\rm)} on $[a,b]$ if$f(qx)\le f(x)$ {\rm (}$f(qx)\ge f(x)${\rm)}
for $x,\,qx \in [a,b]$.
\end{defn}

It is easy to see that if the function $f$ is increasing {\rm (}decreasing), then it is
$q$-increasing {\rm (}$q$-decreasing{\rm)} too.

\section{Results and discussions}

Our main results are contained in three theorems.

\begin{thm}\label{t3.1} Let $f(x)$ be $q$-decreasing function on $[a,b]$ with $f(bq^0)=0$.
Then, for any $p\ge 0$, there holds
\begin{eqnarray}\label{52}
\int_a^b|D_q f(x)||f(x)|^p d_q x\leq(b-a)^p\int_a^b|D_q f(x)|^{p+1}d_q x.
\end{eqnarray}
\end{thm}

\begin{proof} Using Definition \ref{df1} and \eqref{56}, we have
\begin{align*}
\int_a^b|D_q f(x)||f(x)|^p d_q x&=\int_{bq^n}^b\left|\frac{f(x)-f(qx)}{x-qx}\right||f(x)|^p d_q x\\
&=b(1-q)\sum_{j=0}^{n-1}q^j\left|\frac{f(bq^j)-f(bq^{j+1}}{bq^j-bq^{j+1}}\right||f(bq^j)|^p,
\end{align*}
whence, taking into account that $f(x)$ is $q$-decreasing, we have
$$
\sum_{j=0}^{n-1}{|f(bq^j)-f(bq^{j+1})|}|f(bq^j)|^p
\le|f(bq^n)|^p\sum_{j=0}^{n-1}{|f(bq^j)-f(bq^{j+1})|}.
$$
In view of $f(bq^n)=\sum\limits_{j=0}^{n-1}{f(bq^{j+1})-f(bq^j)}$, we obtain
$$
|f(bq^n)|^p=\Big|\sum_{j=0}^{n-1}f(bq^{j+1})-f(bq^j)\Big|^p\leq\Big(\sum_{j=0}^{n-1}\big|f(bq^j)-f(bq^{j+1})\big|\Big)^p,
$$
so that
$$
|f(bq^n)|^p\sum_{j=0}^{n-1}\big|f(bq^j)-f(bq^{j+1})\big|
\leq\Big(\sum_{j=0}^{n-1}\big|f(bq^j)-f(bq^{j+1})\big|\Big)^{p+1}.
$$
Thus
\begin{equation}\label{step1}
\int_a^b|D_q f(x)||f(x)|^p d_q x\leq\Big(\sum_{j=0}^{n-1}\big|f(bq^j)-f(bq^{j+1})\big|\Big)^{p+1}.
\end{equation}
The right-hand side of this inequality we can write in the form of
\begin{align*}
\Big(\sum_{j=0}^{n-1}\big|f(bq^j)-f(bq^{j+1})\big|\Big)^{p+1}
&=\Big(\sum_{j=0}^{n-1}\big|bq^j-bq^{j+1}\big|\Big|\frac{f(bq^j)-f(bq^{j+1})}{bq^j-bq^{j+1}}\Big|\Big)^{p+1}\\
&=\Big(b(1-q)\sum_{j=0}^{n-1}q^j\Big|\frac{f(bq^j)-f(bq^{j+1})}{bq^j-bq^{j+1}}\Big|\Big)^{p+1}.
\end{align*}
After rewriting $q^j=(q^j)^{\frac p{p+1}}(q^j)^{\frac1{p+1}}$, and applying H\"older's inequality
to the last sum, we have
\begin{multline*}
\sum_{j=0}^{n-1}(q^j)^{\frac p{p+1}}(q^j)^{\frac1{p+1}}\Big|\frac{f(bq^j)-f(bq^{j+1})}{bq^j-bq^{j+1}}\Big|\leq
\Big(\sum_{j=0}^{n-1}\big((q^j)^{\frac p{p+1}}\big)^{\frac{p+1}p}\Big)^{\frac p{p+1}}\times\\
\Big(\sum_{j=0}^{n-1}q^j\Big|\frac{f(bq^j)-f(bq^{j+1})}{bq^j-bq^{j+1}}\Big|^{p+1}\Big)^{\frac1{p+1}}.
\end{multline*}
After raising both sides to the power $p+1$, we find
$$
\Big(\sum_{j=0}^{n-1}q^j\Big|\frac{f(bq^j)-f(bq^{j+1})}{bq^j-bq^{j+1}}\Big|\Big)^{p+1}\leq
\Big(\sum_{j=0}^{n-1}q^j\Big)^p\sum_{j=0}^{n-1}q^j\Big|\frac{f(bq^j)-f(bq^{j+1})}{bq^j-bq^{j+1}}\Big|^{p+1}.
$$
Multiplying this inequality by $b^{p+1}(1-q)^{p+1}$, and relying on the formula for the sum of the first $n$ terms of the geometric series, we arrive at the inequality
\begin{multline}\label{p}
\Big(b(1-q)\sum_{j=0}^{n-1}q^j\Big|\frac{f(bq^j)-f(bq^{j+1})}{bq^j-bq^{j+1}}\Big|\Big)^{p+1}\leq
b^p(1-q^n)^p\times\\
b(1-q)\sum_{j=0}^{n-1}q^j\Big|\frac{f(bq^j)-f(bq^{j+1})}{bq^j-bq^{j+1}}\Big|^{p+1}.
\end{multline}
Considering that $b^p(1-q^n)^p=(b-bq^n)^p=(b-a)^p$, taking into account \eqref{step1},
we have proved the inequality
$$
\int_a^b|D_q f(x)||f(x)|^p d_q x\leq(b-a)^p\Big(b(1-q)
\sum_{j=0}^{n-1}q^j\Big|\frac{f(bq^j)-f(bq^{j+1})}{bq^j-bq^{j+1}}\Big|^{p+1}\Big).
$$
Referring to \eqref{56}, there holds
$$
b(1-q)\sum_{j=0}^{n-1}q^j\Big|\frac{f(bq^j)-f(bq^{j+1})}{bq^j-bq^{j+1}}\Big|^{p+1}
=\int_a^b|D_q f(x)|^{p+1}d_q x,
$$
whereby we prove the theorem.
\end{proof}

\begin{rem} In particular, by taking $p=1$, the inequality \eqref{52} in
{\rm Theorem \ref{t3.1}} reduces to the following Opial's inequality in $q$-Calculus.
\begin{eqnarray*}
\int_a^b|D_q f(x)||f(x)|d_q x\leq(b-a)\int_a^b|D_q f(x)|^2 d_q x.
\end{eqnarray*}
\end{rem}

The following theorems are concerned with $q$-monotonic functions.

\begin{thm}\label{T4} If $f(x)$ and $g(x)$ are $q$-decreasing functions on $[a,b]$
satisfying $f(bq^0)=0$ and $g(bq^0)=0$, then there holds the inequality
\begin{multline}\label{53}
\int_a^b\big(f(x)D_q g(x)+g(qx)D_q f(x)\big)d_q x\\
\leq\frac{b-a}2\int_a^b\big((D_q f(x))^2+(D_q g(x))^2\big)d_q x.
\end{multline}
\end{thm}

\begin{proof}Replacing (\ref{q-der}) in the integral
$$
\int_a^b\big(f(x)D_q g(x)+ g(qx)D_q f(x)\big)d_q x,
$$
we obtain
$$
\int_{bq^n}^b\left(f(x)\frac{g(x)-g(qx)}{x-qx}+g(qx)\frac{f(x)-f(qx)}{x-qx}\right)d_q x,
$$
whence, using Gauchman $q$-restricted integral, we have
{\arraycolsep0.05em
\begin{eqnarray*}
b(1-q)&&\Big(\sum_{j=0}^{n-1}q^j f(bq^j)\frac{g(bq^j)-g(bq^{j+1})}{bq^j-bq^{j+1}}
+\sum_{j=0}^{n-1}q^j g(bq^{j+1})\frac{f(bq^j)-f(bq^{j+1})}{bq^j-bq^{j+1}}\Big)\\
&&=\sum_{j=0}^{n-1}\big(f(bq^j)(g(bq^j)-g(bq^{j+1}))+g(bq^{j+1})(f(bq^j)-f(bq^{j+1}))\big)\\
&&=\sum_{j=0}^{n-1}\big(f(bq^j)g(bq^j)-g(bq^{j+1})f(bq^{j+1})\big)=-f(bq^n)g(bq^n).
\end{eqnarray*}
}Using the elementary inequality $-xy\leq\frac12(x^2+y^2),\,x,y\in\mathbb R$, and considering that
$$
f(bq^n)=\sum_{j=0}^{n-1}\big(f(bq^{j+1})-f(bq^j)\big),\quad g(bq^n)=\sum_{j=0}^{n-1}\big(g(bq^{j+1})-g(bq^j)\big),
$$
we find
$$
-f(bq^n)g(bq^n)\leq\frac12\Bigg(\Big(\sum_{j=0}^{n-1}
\big(f(bq^{j+1})-f(bq^j)\big)\Big)^2+\Big(\sum_{j=0}^{n-1}\big(g(bq^{j+1})-g(bq^j)\big)\Big)^2\Bigg)
$$
Applying \eqref{p} for $p=1$, knowing that $f(x)$ and $g(x)$ are $q$-decreasing,
we obtain
$$
\Big(b(1-q)\sum_{j=0}^{n-1}q^j\frac{f(bq^{j+1})-f(bq^j)}{bq^j-bq^{j+1}}\Big)^2\leq
b(1-q^n)b(1-q)\sum_{j=0}^{n-1}q^j\Big(\frac{f(bq^j)-f(bq^{j+1})}{bq^j-bq^{j+1}}\Big)^2
$$
as well as
$$
\Big(b(1-q)\sum_{j=0}^{n-1}q^j\frac{g(bq^{j+1})-g(bq^j)}{bq^j-bq^{j+1}}\Big)^2\leq
b(1-q^n)b(1-q)\sum_{j=0}^{n-1}q^j\Big(\frac{g(bq^j)-f(bq^{j+1})}{bq^j-bq^{j+1}}\Big)^2.
$$
Since $b(1-q^n)=b-a$, making use of \eqref{56}, we have
$$
\int_a^b\big(f(x)D_q g(x)+g(qx)D_q f(x)\big)d_q x\leq
\frac{b-a}2\int_a^b\big((D_q f(x))^2+(D_q g(x))^2\big)d_q x,
$$
whereby (\ref{53}) is proved.
\end{proof}

\begin{thm} If $f(x)$ and $g(x)$ are $q$-decreasing functions on $[a,b]$
satisfying $f(bq^0)=g(bq^0)=0$, then there holds the inequality
\begin{multline}\label{51}
\int_a^b|f(x)|^s|g(x)|^t d_q x\\
\leq(b-a)^{s+t}\int_a^b\Big(\frac s{s+t}|D_q f(x)|^{s+t}d_q x+\frac t{s+t}|D_q g(x)|^{s+t}d_q x\Big).
\end{multline}
\end{thm}

\begin{proof} First, we apply \eqref{56} to the left-hand side of \eqref{51}, and have
$$
\int_a^b|f(x)|^s|g(x)|^t d_q x=b(1-q)\sum_{i=0}^{n-1}q^i|f(bq^i)|^s|g(bq^i)|^t.
$$
For real numbers $z,w\geq0$ and $s,t>0$, we rely on the elementary inequality
$$
z^s w^t\leq\frac s{s+t}z^{s+t}+\frac t{s+t}w^{s+t}.
$$
After setting $z=(q^i)^{\frac1{s+t}}|f(bq^i)|,\,w=(q^i)^{\frac1{s+t}}|g(bq^i)|$, we find
\begin{multline*}
\sum_{i=0}^{n-1}q^i|f(bq^i)|^s|g(bq^i)|^t=\sum_{i=0}^{n-1}\big((q^i)^{\frac1{s+t}}|f(bq^i)|\big)^s
\big((q^i)^{\frac1{s+t}}|g(bq^i)|\big)^t\\
\leq\frac s{s+t}\sum_{i=0}^{n-1}q^i|f(bq^i)|^{s+t}+\frac t{s+t}\sum_{i=0}^{n-1}q^i|g(bq^i)|^{s+t}.
\end{multline*}
Considering that $f$ and $g$ are $q$-decreasing functions, so $|f(bq^i)|^{s+t}\leq|f(bq^n)|^{s+t}$ and
$|g(bq^i)|^{s+t}\leq|g(bq^n)|^{s+t}$, the last inequality becomes
$$
\sum_{i=0}^{n-1}q^i|f(bq^i)|^s|g(bq^i)|^t\leq\frac{1-q^n}{1-q}\Big(\frac s{s+t}|f(bq^n)|^{s+t}
+\frac t{s+t}|g(bq^n)|^{s+t}\Big).
$$
However, there holds
\begin{eqnarray*}
|f(bq^n)|^{s+t}=\Big|\sum\limits_{i=0}^{n-1}f(bq^{i+1})-f(bq^i)\Big|^{s+t}
\leq\Big(\sum\limits_{i=0}^{n-1}|f(bq^{i+1})-f(bq^i)|\Big)^{s+t}\\
|g(bq^n)|^{s+t}=\Big|\sum\limits_{i=0}^{n-1}g(bq^{i+1})-g(bq^i)\Big|^{s+t}
\leq\Big(\sum\limits_{i=0}^{n-1}|g(bq^{i+1})-g(bq^i)|\Big)^{s+t},
\end{eqnarray*}
so that we have
\begin{multline}\label{3}
\int_a^b|f(x)|^s|g(x)|^t d_q x=b(1-q)\sum_{i=0}^{n-1}q^i|f(bq^i)|^s|g(bq^i)|^t\\
\leq b(1-q^n)\Bigg(\frac s{s+t}\Big(\sum_{i=0}^{n-1}|f(bq^{i+1})-f(bq^i)|\Big)^{s+t}\\
\hskip5cm+\frac t{s+t}\Big(\sum_{i=0}^{n-1}|g(bq^{i+1})-g(bq^i)|\Big)^{s+t}\Bigg)\\
\hskip1cm\leq\frac{b(1-q^n)}{(s+t)}
\Big[s\Big(b(1-q)\sum_{i=0}^{n-1}q^i\Big|\frac{f(bq^i)-f(bq^{i+1})}{bq^i-bq^{i+1}}\Big|\Big)^{s+t}\\
+t\Big(b(1-q)\sum_{i=0}^{n-1}q^i\Big|\frac{g(bq^i)-g(bq^{i+1})}{bq^i-bq^{i+1}}\Big|\Big)^{s+t}\Big],
\end{multline}

Here we follow the same procedure as in the proof of Theorem \ref{t3.1}.
So, after rewriting $q^i=(q^i)^{\frac{s+t-1}{s+t}}(q^j)^{\frac1{s+t}}$, and applying H\"older's
inequality to both sums on the right side of the last inequality, for the first sum we have
\begin{multline*}
\sum_{j=0}^{n-1}(q^j)^{\frac{s+t-1}{s+t}}(q^j)^{\frac1{s+t}}\Big|\frac{f(bq^j)-f(bq^{j+1})}{bq^j-bq^{j+1}}\Big|\leq
\Big(\sum_{j=0}^{n-1}\big((q^j)^{\frac{s+t-1}{s+t}}\big)^{\frac{s+t}{s+t-1}}\Big)^{\frac{s+t-1}{s+t}}\times\\
\Big(\sum_{j=0}^{n-1}q^j\Big|\frac{f(bq^j)-f(bq^{j+1})}{bq^j-bq^{j+1}}\Big|^{s+t}\Big)^{\frac1{s+t}},
\end{multline*}
and for the second as well
\begin{multline*}
\sum_{j=0}^{n-1}(q^j)^{\frac{s+t-1}{s+t}}(q^j)^{\frac1{s+t}}\Big|\frac{g(bq^j)-g(bq^{j+1})}{bq^j-bq^{j+1}}\Big|\leq
\Big(\sum_{j=0}^{n-1}\big((q^j)^{\frac{s+t-1}{s+t}}\big)^{\frac{s+t}{s+t-1}}\Big)^{\frac{s+t-1}{s+t}}\times\\
\Big(\sum_{j=0}^{n-1}q^j\Big|\frac{g(bq^j)-g(bq^{j+1})}{bq^j-bq^{j+1}}\Big|^{s+t}\Big)^{\frac1{s+t}}.
\end{multline*}
We multiply both inequalities by $b(1-q)$, then raise them to the power $s+t$. Thus, we obtain
\begin{multline*}
\hskip-0.35cm \Big(b(1-q)\sum_{j=0}^{n-1}q^j\Big|\frac{f(bq^j)-f(bq^{j+1})}{bq^j-bq^{j+1}}\Big|\Big)^{s+t}\leq
b^{s+t-1}(1-q^n)^{s+t-1}\times\\
\hskip0.31cm b(1-q)\sum_{j=0}^{n-1}q^j\Big|\frac{f(bq^j)-f(bq^{j+1})}{bq^j-bq^{j+1}}\Big|^{s+t}
=b^{s+t-1}(1-q^n)^{s+t-1}\int_a^b|D_q f(x)|^{s+t}d_q x,
\end{multline*}
and similarly
\begin{multline*}
\hskip-0.35cm \Big(b(1-q)\sum_{j=0}^{n-1}q^j\Big|\frac{g(bq^j)-g(bq^{j+1})}{bq^j-bq^{j+1}}\Big|\Big)^{s+t}\leq
b^{s+t-1}(1-q^n)^{s+t-1}\times\\
\hskip0.31cm b(1-q)\sum_{j=0}^{n-1}q^j\Big|\frac{g(bq^j)-g(bq^{j+1})}{bq^j-bq^{j+1}}\Big|^{s+t}
=b^{s+t-1}(1-q^n)^{s+t-1}\int_a^b|D_q g(x)|^{s+t}d_q x,
\end{multline*}
so that, because $b(1-q^n)=b-a$, from \eqref{3} there follows \eqref{51}, whereby we complete the proof.
\end{proof}

\begin{rem} In the special case when $s=t=r$ and $f(x)=g(x)=h(x)$, the inequality
established in \eqref{51} reduces to the $q$-Wirtinger-type inequality
$$
\int_a^b {|h(x)|^{2r} d_q x}\leq(b-a)^{2r}\int_a^b {(D_q h(x))^{2r} d_q x}.
$$
\end{rem}

\end{document}